\documentclass{article}

\usepackage{graphicx}
\usepackage{amsmath,amsthm,amssymb}
\usepackage{tabu}
\newtheorem{theorem}{Theorem}[section]
\newtheorem{lemma}[theorem]{Lemma}

\theoremstyle{definition}

\theoremstyle{remark}
\newtheorem{remark}[theorem]{Remark}

\numberwithin{equation}{section}
\DeclareMathOperator{\isOdd}{ is \ odd}
\DeclareMathOperator{\isEven}{ is \ even}

\begin{document}

\title{Parity of the Partition Function $p(n,k)$}
\author{Kedar Karhadkar}

\maketitle

\begin{abstract}
Let $p(n, k)$ denote the number of partitions of $n$ into parts less than or equal to $k$. We show several properties of this function modulo 2. First, we prove that for fixed positive integers $k$ and $m$, $p(n,k)$ is periodic modulo $m$. Using this, we are able to find lower and upper bounds for the number of odd values of the function for a fixed $k$.
\end{abstract}

\section{Introduction}
A partition of a positive integer $n$ is a non-increasing sequence of positive integers whose sum equals $n$. For instance, there are  five partitions of $4$: $4, 3 + 1, 2 + 2, 2 + 1 + 1$, and $1+1+1+1$. Let $p(n)$ denote the number of partitions of $n$ (For more details on this function, see \cite{Andrews}). It has been conjectured that the odd density of $p(n)$ is $\frac{1}{2}$. That is,\\
$$\lim_{N \rightarrow \infty} \frac{\#\{n \leq N | p(n) \isOdd\}}{N} = \frac{1}{2}.$$

Although this problem remains open, several properties of the distribution of $p(n)$ modulo 2 have been discovered. In 1959, Newman \cite{Newman} showed that $p(n)$ achieves both odd and even values infinitely often, and is not periodic modulo any integer. Nicolas, Ruzsa, and S\'ark\"ozy \cite{Nicolas}  showed that the partition function satisfies
$$ \sqrt{N} \ll \#\{n \leq N | p(n) \isEven\}.$$
Ahlgren \cite{Ahlgren} found similar bounds when taking $n$ in arithmetic progressions.\\
\indent The strongest bounds known currently still do not preclude the possibility of the odd density of $p(n)$ being zero. However, the efforts made on the parity of $p(n)$ motivate a similar discussion for $p(n, k)$, the number of partitions of $n$ into parts less than or equal to $k$. It is not difficult to see that the generating function for $p(n, k)$ is given by
$$\sum_{n=0}^{\infty}p(n,k)q^n = \prod_{n=1}^{k}\frac{1}{1-q^n}.$$

Since the generating function is a finite product rather than an infinite product, $p(n, k)$ is in some ways more convenient to work with than the ordinary partition function. Kronholm \cite{Kronholm} found that the function modulo odd primes $k$ yields congruences of the form
$$p(nk, k) - p(nk - \text{lcm}(1, 2, \cdots, k), k) \equiv 0 \pmod k$$
for sufficiently large $n$.

In this paper, we will find new properties of $p(n, k)$. The most important result we will need, stated in the following theorem, reduces the odd density of $p(n, k)$ to a finite problem.
\begin{theorem} \label{thm1.1}
For any positive integers $k$ and $m$, the function $p(n, k)$ modulo $m$ is periodic on $n$. That is, given $k, m \in \mathbb{N}$, there exists $L \in \mathbb{N}$ such that $p(n, k) \equiv p(n + L, k) \pmod m$ for all values of $n$.
\end{theorem}
Given this theorem, we can explicitly compute the odd density of $p(n, k)$ for small values of $k$. Additionally, it allows us to find an upper bound for the odd density in Section 3.
\begin{theorem} \label{thm1.2}
There exist infinitely many values of $k$ such that
\end{theorem}
$$\lim_{N \rightarrow \infty}\frac{\#\{n \leq N | p(n, k) \isOdd\}}{N} \leq \frac{2}{3}.$$

In the opposite direction, we can also look at the maximum consecutive number of even values of $p(n, k)$. By observing how this number varies depending on $k$, we find a weaker lower bound in Section 4.
\begin{theorem} \label{thm1.3}
For all positive integers $k$,
$$\lim_{N \rightarrow \infty}\frac{\#\{n \leq N | p(n, k) \isOdd\}}{N} \geq \frac{2}{k(k+1)}.$$
\end{theorem}

\section{Periodicity of $p(n, k)$ modulo $m$}

\begin{lemma}
For any integer $n$ and any positive integer $ k >1$, 
 \begin{equation}
 p(n, k) = p(n - k, k) + p(n, k - 1), \label{eq1}
 \end{equation}
 where $p(0,k)=1$ and $p(n,k)=0$ if $n<0$. 
\end{lemma}
\begin{proof}
For any $k > 1$, $p(n - k, k)$ counts the number of partitions of $n$ into parts of size at most $k$ with at least one part of size $k$. On the other hand, $p(n, k - 1)$ counts the number of partitions of $n$ into parts of size at most $k$ with no parts of size $k$. Summing both together, the identity follows.
\end{proof}

\subsection{Proof of Theorem~\ref{thm1.1}.}
We prove Theorem~\ref{thm1.1} by induction on $k$. 

For $k = 1$, the statement clearly holds since $p(n, 1) = 1 \equiv 1 \pmod m$. Hence $p(n, 1)$ has period $1$. 

Now, suppose that $p(n, k)$ is periodic with period $L$. We claim that $p(n, k + 1)$ has period $m(k+1)L$. To see this, note that for all positive integers $n$,
\begin{align*}
   & p(n + (k+1)L, k + 1) - p(n, k + 1) \\
    &= \sum_{i = 1}^{L}{p(n + (k + 1)i, k + 1) - p(n + (k + 1)(i-1), k + 1)}.
\end{align*}
By \eqref{eq1}, we can rewrite the above equation as
\begin{align}
    p(n + (k+1)L, k + 1) - p(n, k + 1) &= \sum_{i = 1}^{L}{p(n + (k + 1)i, k)}.
    \label{eq2}    
\end{align}
Since \eqref{eq2} holds for all positive integers $n$, for any non-negative integer $a$, we can substitute $n + a(k + 1)L$ for $n$:
\begin{align*}
  &  p(n + (a + 1)(k+1)L, k + 1) - p(n + a(k+1)L, k + 1) \\
  &= \sum_{i = 1}^{L}{p(n + a(k+1)L + (k + 1)i, k)}.
\end{align*}
But since $p(n, k)$ modulo $m$ has period $L$, we can simplify this to
\begin{align*}
    & p(n + (a + 1)(k+1)L, k + 1) - p(n + a(k+1)L, k + 1) \\
    &\equiv \sum_{i = 1}^{L}{p(n + (k + 1)i, k)} \pmod m,
\end{align*}
which we observe does not depend on $a$.
Summing over $a$, we obtain
\begin{align*}
&\indent p(n + m(k + 1)L, k + 1) - p(n, k + 1)\\
&=\sum_{a = 0}^{m-1}{p(n + (a+1)(k + 1)L, k + 1) - p(n + a(k + 1)L, k + 1)}\\
&\equiv\sum_{a = 0}^{m-1}\sum_{i = 1}^{L}{p(n + (k + 1)i, k)} \pmod m\\
&\equiv 0 \pmod m.
\end{align*}

So, $p(n, k + 1)$ is periodic with period $m(k + 1)L$, as desired. The result follows by induction. More directly, we can take the period of $p(n, k)$ to be $m^{k - 1}k!$.

\begin{remark}
We can use the periodicity of $p(n, k)$ modulo 2 to determine the odd density of the function, simply by checking the parity of $p(n, k)$ for $2^{k-1}k!$ consecutive values of $n$. With the use of a computer, this approach gives us the densities for the first $10$ values of $k$, which are shown in the following table.
\end{remark}
\begin{center}
\begin{table}[h]
\begin{tabular}{|l|c|c|c|c|c|c|c|c|c|c|}
\hline
$k$ & 1 & 2 & 3 & 4 & 5 & 6 & 7 & 8 & 9 & 10 \\ \hline
Odd Density of $p(n, k)$ & 1 & $\frac{1}{2}$& $\frac{5}{12}$&$\frac{11}{24}$&$\frac{1}{2}$& $\frac{29}{60}$&$\frac{23}{56}$&$\frac{1}{2}$& $\frac{27}{56}$& $\frac{1}{2}$    \\ \hline
\end{tabular}
\end{table}
\end{center}
\section{Upper bound for Odd Density of $p(n, k)$}
Let $k$ be fixed. Because $p(n, k)$ is periodic modulo 2, its odd density exists and is determined by the number of odd values of $p(n, k)$ over a finite range of $n$. We can use this fact to easily relate the number of odd values of $p(n, k - 1)$ to the number of odd values of $p(n, k)$ over $n$.
\begin{lemma} \label{lemma2} For all integers $k > 1$, $p(n, k)$
 satisfies
 \begin{align*}
& \lim_{N \rightarrow \infty}\frac{\#\{n \leq N | p(n, k) \isOdd, p(n, k - 1) \isOdd\}}{N}\\
& =\lim_{N \rightarrow \infty}\frac{\#\{n \leq N | p(n, k) \isEven, p(n, k - 1) \isOdd \}}{N}.
 \end{align*}
\end{lemma}
\begin{proof}
Let $k \in \mathbb{N}$ be fixed. Consider
\begin{align*}
&\indent\lim_{N \rightarrow \infty}\frac{\#\{n \leq N | p(n, k) \isOdd\}}{N}\\&= \lim_{N \rightarrow \infty}\frac{\#\{n \leq N |p(n, k)\isOdd, p(n, k - 1) \isEven\}}{N}\\ &+\lim_{N \rightarrow \infty}\frac{\#\{n \leq N |p(n, k)\isOdd, p(n, k - 1) \isOdd\}}{N}.
\end{align*}
By \eqref{eq1}, we can simplify this to
\begin{align*}
&\indent\lim_{N \rightarrow \infty}\frac{\#\{n \leq N | p(n, k) \isOdd\}}{N}\\&=
    \lim_{N \rightarrow \infty}\frac{\#\{n \leq N | p(n + k, k) - p(n + k, k - 1) \isOdd\}}{N}.
\end{align*}
Since we are taking $N \rightarrow \infty$, we can replace instances of $n + k$ with $n$.
\begin{align*}
    &=\lim_{N \rightarrow \infty}\frac{\#\{n \leq N |p(n, k) - p(n, k - 1) \isOdd\}}{N}\\
    &=\lim_{N \rightarrow \infty}\frac{\#\{n \leq N |p(n, k)\isOdd, p(n, k - 1) \isEven\}}{N}\\
    &+\lim_{N \rightarrow \infty}\frac{\#\{n \leq N |p(n, k)\isEven, p(n, k - 1) \isOdd\}}{N}.
\end{align*}
By comparing this to the original expression and cancelling terms, the desired result follows.
\end{proof}
Let $k$ be a fixed positive integer. Lemma \ref{lemma2} tells us that for every two odd values of $p(n, k - 1)$, there is one even value of $p(n, k)$. For instance, if $p(n, k - 1)$ is odd for all $n$, then $p(n, k)$ will be odd for half of the values of $n$. In this sense, odd values are ``better behaved'' than even values. We can predict the odd density of $p(n, k)$ using the number of odd values of $p(n, k - 1)$, while even values give no information. Using this property, we can set a bound for the odd density of $p(n, k)$ for a fixed $k$.

\subsection{Proof of Theorem~\ref{thm1.2}}
We will prove the result by showing that if the odd density of $p(n, k)$ is greater than $\frac{2}{3}$ for some $k$, the odd density of $p(n, k + 1)$ must be less than $\frac{2}{3}$. Hence the odd density would have to be at most $\frac{2}{3}$ for infinitely many $k$.

Suppose that $$\lim_{N \rightarrow \infty}\frac{\#\{n \leq N | p(n, k) \isOdd\}}{N} > \frac{2}{3}.$$
Then we claim that the odd density in $p(n, k + 1)$ is less than or equal to $\frac{2}{3}$. To see this, we apply Lemma 3.1 to rewrite the inequality as
\begin{align*}
2\lim_{N \rightarrow \infty}\frac{\#\{n \leq N | p(n, k) \isOdd, p(n, k + 1) \isEven\}}{N}& > \frac{2}{3}\\
\lim_{N \rightarrow \infty}\frac{\#\{n \leq N | p(n, k) \isOdd, p(n, k + 1) \isEven\}}{N}& > \frac{1}{3}.
\end{align*}
So certainly,
$$ \lim_{N \rightarrow \infty}\frac{\#\{n \leq N | p(n, k + 1) \isEven\}}{N} > \frac{1}{3}.$$
Therefore, there exist infinitely many values of $k$ for which the odd density is less than or equal to $\frac{2}{3}$.

\section{Lower Bound for Odd Density of $p(n, k)$}
Let $k$ be a fixed positive integer. By Theorem~\ref{thm1.1}, we know that $p(n, k)$ is periodic modulo 2,
so we can write its generating function as
$$
\sum_{n=0}^{\infty}p(n,k)q^n = \prod_{n=1}^{k}\frac{1}{(1-q^n)} \equiv \frac{a(q)}{1 - q^L} \pmod 2,
$$
where $L$ is the period of $p(n,k)$ and $a(q)$ is a polynomial with degree less than $L$. Rearranging to eliminate the denominators, we obtain
\begin{align*}
a(q)\prod_{n=1}^{k}(1-q^n) &\equiv 1-q^L \pmod 2 ,
\end{align*}
from which we see that
\begin{align*}
\deg a(q) &= L - \frac{k(k + 1)}{2}
\end{align*}
But this means that the last $\frac{k(k+1)}{2} - 1$ terms of a period of $p(n, k)$ must be zero modulo $2$, since the last nonzero coefficient of $a(q)$ is that of $a^{L - \frac{k(k+1)}{2}}$. So for any fixed $k$, there exist $\frac{k(k+1)}{2} - 1$ consecutive values of $n$ such that $p(n, k)$ is divisible by $2$. We shall show that this is the longest consecutive string of zeroes.
\begin{theorem}\label{thm4}
For any positive integer $k$, there exist at most $\frac{k(k+1)}{2} - 1$ consecutive values of $n$ such that $p(n, k)$ is even.
\end{theorem}
\begin{proof}
We prove by induction on $k$. If $k = 1$, then $p(n, 1) = 1$ for all positive integers $n$, so the statement is true. This establishes the base case. Now, suppose that the statement is true for $k = j - 1$. Then, if
\begin{align*}
p(i, j), p(i + 1, j), \cdots p(i + j(j + 1)/2 - 2, j)
\end{align*}
are all even, then
\begin{align*}
p(i + j, j) - p(i, j), p(i + 1 + j, j) - p(i + 1, j), \cdots,\\ p(i + j(j + 1) / 2 - 2, j) - p(i+ j(j + 1)/2 - 2 - j, j)
\end{align*}
are all even. Stated differently, this implies that
\begin{align*}
p_{}(i + j, j - 1), p(i + j + 1, j - 1), \cdots, p(i + j(j + 1) / 2 - 2, j - 1)
\end{align*}are all even. This is a sequence of $\frac{(j-1)j}{2} - 1$ consecutive even numbers. So, by the inductive hypothesis, $p(i + j(j + 1)/2 - 1, j - 1)$ must be odd. But we already know that $$p(i + j(j + 1)/2 - 1 - j, j)$$ is even, so
\begin{align*}
p(i + j(j + 1)/2 - 1, j) = p(i + j(j + 1)/2 - 1, j - 1) + p(i + j(j + 1)/2 - 1 - j, j)
\end{align*}
is odd. The result follows by induction.
\end{proof}
The preceding theorem immediately yields a lower bound for the odd density of $p(n, k)$.

\subsection{Proof of Theorem~\ref{thm1.3}}

Fix a positive integer $k$ and consider the sequence of terms of $p(n, k)$. By Theorem \ref{thm4}, there can be at most $\frac{k(k+1)}{2} - 1$ even terms in a row. Therefore, among every  $\frac{k(k+1)}{2}$ consecutive terms of $p(n, k)$, there must be at least one odd term. Thus, $$\lim_{N \rightarrow \infty}\frac{\#\{n \leq N | p(n, k) \isOdd\}}{N} \geq \frac{2}{k(k+1)}.$$ 
\vfill\eject
\section*{ Acknowledgements}
The author would like to thank Ae Ja Yee for her guidance and time.

\end{document}